\documentclass[12pt,english,amsthm]{amsart}
\usepackage[T1]{fontenc}
\usepackage[latin9]{inputenc}
\usepackage{geometry}
\usepackage{units}
\usepackage{amsthm}
\usepackage{amstext}
\usepackage{amssymb}

\makeatletter

\providecommand{\tabularnewline}{\\}

\numberwithin{equation}{section}
\numberwithin{figure}{section}
\theoremstyle{plain}
\newtheorem{thm}{\protect\theoremname}
  \theoremstyle{definition}
  \newtheorem{defn}[thm]{\protect\definitionname}
  \theoremstyle{plain}
  \newtheorem{lem}[thm]{\protect\lemmaname}
  \theoremstyle{plain}
  \newtheorem{cor}[thm]{\protect\corollaryname}
  \theoremstyle{remark}
  \newtheorem{rem}[thm]{\protect\remarkname}
  \theoremstyle{plain}
  \newtheorem{prop}[thm]{\protect\propositionname}
  \theoremstyle{plain}
  \newtheorem{lyxalgorithm}[thm]{\protect\algorithmname}
  \theoremstyle{remark}
  \newtheorem*{acknowledgement*}{\protect\acknowledgementname}
  \theoremstyle{definition}
  \newtheorem{example}[thm]{\protect\examplename}

\usepackage{nicefrac}
\usepackage{url}
\usepackage{enumerate}

\def\cocoa{{\hbox{\rm C\kern-.13em o\kern-.07em C\kern-.13em o\kern-.15em A}}}

\newcommand{\N}{\mathbb{N}}

\newcommand{\T}{\mathbb{T}}
\newcommand{\Tnm}{\mathbb{T}^n_m}
\newcommand{\F}{\mathcal{F}}
\renewcommand{\S}{\mathcal{S}}

\DeclareMathOperator{\NS}{NS}
\DeclareMathOperator{\LT}{LT}
\DeclareMathOperator{\LC}{LC}
\DeclareMathOperator{\lcm}{lcm}

\DeclareMathOperator{\syz}{Syz}
\DeclareMathOperator{\psyz}{PSyz}
\DeclareMathOperator{\Spol}{Spol}

\newcommand\sigvar{e}

\makeatother

\usepackage{babel}
  \providecommand{\acknowledgementname}{Acknowledgement}
  \providecommand{\algorithmname}{Algorithm}
  \providecommand{\corollaryname}{Corollary}
  \providecommand{\definitionname}{Definition}
  \providecommand{\examplename}{Example}
  \providecommand{\lemmaname}{Lemma}
  \providecommand{\propositionname}{Proposition}
  \providecommand{\remarkname}{Remark}
\providecommand{\theoremname}{Theorem}

\begin{document}

\title{The F5 Criterion revised}

\author{Alberto Arri}

\curraddr{Google Corporation}

\address{Scuola Normale Superiore di Pisa - Piazza dei Cavalieri, 7 - 56126
Pisa, Italy}

\email{arri@sns.it}

\author{John Perry}

\address{University of Southern Mississippi, Hattiesburg, MS USA}

\email{john.perry@usm.edu}
\begin{abstract}
The purpose of this work is to generalize part of the theory behind
Faugère's ``F5'' algorithm. This is one of the fastest known algorithms
to compute a Gröbner basis of a polynomial ideal $I$ generated by
polynomials $f_{1},\ldots,f_{m}$. A major reason for this is what
Faug\`ere called the algorithm's ``new'' criterion, and we call
``the F5 criterion''; it provides a sufficient condition for a set
of polynomials $G$ to be a Gröbner basis. However, the F5 algorithm
is difficult to grasp, and there are unresolved questions regarding
its termination.

This paper introduces some new concepts that place the criterion in
a more general setting: $\S$-Gr\"obner bases and primitive $\S$-irreducible
polynomials. We use these to propose a new, simple algorithm based
on a revised F5 criterion. The new concepts also enable us to remove
various restrictions, such as proving termination without the requirement
that $f_{1},\ldots,f_{m}$ be a regular sequence.
\end{abstract}

\maketitle

\section{Introduction}

Since their introduction by Buchberger~\cite{GBBIB706}, Gröbner
bases and their computation have attracted significant attention in
the computer algebra community. The best-known algorithm used to compute
a Gröbner basis is the original algorithm due to Buchberger, and named
after him. Its efficiency has been constantly enhanced through the
years, but there remains room for improvement. Various criteria have
since been introduced to detect useless computations -- for example,~\cite{GBBIB706,Buchberger79,GBBIB1064}
--- but even so, the algorithm spends most of its time reducing polynomials
to zero (``zero reductions'').

Lazard~\cite{Lazard83} pointed out that one can view the computation
of a Gr\"obner basis as the reduction to row-echelon form of the
Macaulay matrix of the ideal. This led to the Staggered Linear Basis
algorithm of Gebauer and M\"oller~\cite{StaggeredLinearBases},
as well as the ``F4'' algorithm of Faug\`ere~\cite{Fau99}. M\"oller,
Mora, and Traverso exploited the relationship between zero reductions
and syzygies~\cite{143343}, but although the algorithm they presented
successfully detected many zero reductions, in practice it took too
much memory and time (see Section~8 of~\cite{143343}). Faugère~\cite{f5}
combined aspects of these approaches into algorithm ``F5'', which
for a certain class of polynomial system eliminates \emph{all} zero
reductions. This algorithm exhibits impressive performance.

By Faug\`ere's admission, the theory behind the algorithm's new criterion,
which we call \emph{the F5 criterion}, is merely sketched, so as to
leave more room for examples and an accurate description of the algorithm.
The proof of the algorithm's termination and correctness were likewise
only outlined. Additionally, some arguments were made under strong
assumptions, such as that the input sequence $f_{1},\ldots,f_{m}$
had only principal syzygies (such a sequence is called a \emph{regular}
sequence).

We pause a moment to consider some variants of F5. Bardet described
an implementation of F5 in matrix form, where termination is ensured
by manually supplying a maximal degree~\cite{mb}. Stegers filled
in some details of Faug\`ere's proof in~\cite{cryptoeprint:2006:404},
but stopped at two conjectures, one of which Gash later showed to
be false~\cite{Gash2008}.

The purpose of this paper is to present a simpler algorithm that illustrates
the fundamental principles of F5 without sacrificing termination.
We begin by defining a function $\S$ which is equivalent to that
of Faugère, then develop a structured theory, introducing new concepts
such as \emph{primitive $\S$-irreducible polynomials} and \emph{$\S$-Gröbner
bases}. These make the study of the problem more accessible, and suggest
a new version of the F5 criterion which depends neither on the regularity
of the input, nor on a particular ordering on the module of syzygies.

From this theory, we develop a new, simpler algorithm. We must emphasize
that the algorithm is a simple demonstration of the criterion, and
not a deep treatment of how to implement a highly efficient algorithm;
nevertheless, the new concepts allow us to prove correctness and termination
\emph{for any input}. Note that although some F5-style algorithms
provide explicit termination mechanisms \cite{mb,Gash2008}, these
mechanisms rely on previously-developed, non-F5 criteria to compute
a maximal degree explicitly; by contrast, the termination criterion
used here is precisely the generalized F5 criterion used to detect
useless computations. Later, we show that if we know that the input
is a regular sequence and we use a specific ordering on $\syz\F$,
we can avoid all the reductions to zero. We compare the results to
both F5 and the Staggered Linear Basis algorithm, showing how this
new algorithm differs from each.

The paper's structure is as follows. Sections~\ref{sec:Preliminaries}--\ref{sec:Properties-of-S-reductions}
cover background material; although most of this is relatively straightforward,
an important and novel contribution of the paper appears at the end
of Section~\ref{sec:Properties-of-S-reductions} with Proposition~\ref{prop_finiteSGb}.
The proof of that theorem leads to the concept of \emph{primitive}
$\S$-irreducible polynomials, from which we obtain in Section~\ref{sec:The main result}
a new characterization theorem for a Gr\"obner basis (Theorem~\ref{F5crit}).
In Section~\ref{sec:The-algorithm}, we use this characterization
to formulate the new algorithm, and we prove that it terminates correctly.
Section~\ref{sec:Comparison-with-previous} compares this algorithm
to the Staggered Linear Basis algorithm and F5, illustrating the differences
concretely. Section~\ref{sec:Conclusions-and-future} describes some
conclusions and possible future directions.

\section{\label{sec:Preliminaries}Preliminaries}


Let $P=k[x_{1},\ldots,x_{n}]$ be the polynomial ring over the field
$k$ with $n$ indeterminates, let $\mu$ be any admissible ordering
on $\T^{n}$, the monoid of power products over $x_{1},\ldots,x_{n}$:
$\T^{n}=\left\{ {\textstyle \prod_{i=1}^{n}x_{i}^{\alpha_{i}}\mid\alpha_{i}\in\N}\right\} $.

Let $P^{m}$ be the free $P$-module generated by $\{e_{1},\ldots,e_{m}\}$
and let $\mu'$ be any admissible\footnote{On what we mean by an ``admissible'' ordering, see the appendix.}
ordering on $\Tnm$, the set of module terms of $P^{m}$: $\Tnm=\{te_{\ell}\mid t\in\T,\ell\in\{1,\dots,m\}\}$.

Fix $\F=(f_{1},\ldots,f_{m})\in P^{m}$ and let $I\subseteq P$ be
the ideal generated by $\F$, and define $v:P^{m}\to I$ as the $P$-module
homomorphism such that $v(e_{i})=f_{i}$, and let $\syz\F=\ker v$,
so that $\syz\F$ is the module of syzygies of $\F$, $\LT(\syz\F)\subseteq\Tnm$
is set of leading module terms of $\syz\F$, and $\NS(\syz\F)=\Tnm\setminus\LT(\syz\F)$
is the \emph{normal set} of the syzygies of $\F$.

Clearly $v$ is surjective; therefore, as a $P$-module, $\nicefrac{P^{m}}{\syz\F}\simeq I$.
Let $\psi:I\to\nicefrac{P^{m}}{\syz\F}$ be the $P$-module isomorphism
between them. We use the notation $\LT(\cdot)$ for both the leading
term of a polynomial in $P$ with respect to $\mu$, and the module
leading term of a module element in $P^{m}$ with respect to $\mu'$.
We will use $\LC(f)$, where $f$ is a nonzero polynomial belonging
to $I$, to denote the coefficient of $\LT(f)$.

We are interested in finding a set of polynomials $G$ such that $G$
is a Gröbner basis for $I$ with respect to the ordering $\mu$ on
$\T^{n}$.
\begin{defn}
Let
\[
\begin{array}{lccl}
\S: & I\setminus\{0\} & \to & \NS(\syz\F)\\
 & f & \mapsto & \LT(\psi(f)),
\end{array}
\]
where $\LT(\psi(f))$ is the module leading term of the normal form
of $\psi(f)$ with respect to the ordering $\mu'$ on $P^{m}$. 
\end{defn}
The key idea of Faugère is to keep track of the value of $\S(f)$
for any polynomial $f$ we will work with. It is however clear from
the definition that the explicit calculation of $\S$ requires, at
least, to know a Gröbner basis of $\syz\F$ which is computationally
expensive to compute, more than a Gröbner basis of $I$ itself. In
fact, we will obtain $\S$ from the fact that $\S(f_{i})=e_{i}$ (unless
$e_{i}\in\LT(\syz\F)$) and from other properties of $\S$.

\section{\label{sec: Properties of S}Properties of $\S$}



\begin{lem}
[Properties of $\S$]\label{lemma_prop_S} Let $f,f_{1},f_{2}\in I\setminus\{0\}$.
The following hold: 
\begin{enumerate}
\item \label{lemma_prop_S1} If $\S(f_{1})>\S(f_{2})$ then: 
\[
\S(f_{1}+f_{2})=\S(f_{1}).
\]
 
\item \label{lemma_prop_S2} If $\S(f_{1})=\S(f_{2})=\sigma$ and there
is no $\lambda\in k^{*}=k\setminus\{0\}$ such that $f_{1}=\lambda f_{2}$,
then there exist $\alpha$ and $\beta$ in $k^{*}$ such that: 
\[
\S(\alpha f_{1}+\beta f_{2})<\sigma.
\]
 
\item \label{lemma_prop_S3} Let $t\in\T^{n}$, then 
\begin{align*}
\S(tf)=t\S(f) & \iff t\S(f)\in\NS(\syz\F),\\
\S(tf)<t\S(f) & \iff t\S(f)\in\LT(\syz\F).
\end{align*}
 
\end{enumerate}
\end{lem}
\begin{proof}
(\ref{lemma_prop_S1}) and (\ref{lemma_prop_S2}) are trivial.

In order to prove (\ref{lemma_prop_S3}), let $\S(f)=\sigma=\tau e_{i}$.
By the definition of $\S$ we have: 
\[
f=v\left(\alpha\tau e_{i}+\textrm{smaller terms}\right),
\]
 where $\alpha\in k^{*}$, $\tau\in\T^{n}$, and the argument of $v$
is in its normal form with respect to $\syz\F$.

Multiplying both sides by $t$, we get: 
\[
tf=v\left(\alpha t\tau e_{i}+\textrm{smaller terms}\right).
\]
If $t\tau e_{i}=t\sigma\in\NS(\syz\F)$, the leading term of the normal
form of $\alpha t\tau e_{i}+\cdots$ is $t\sigma$ and, in this case,
$\S(tf)=t\sigma=t\S(f)$. Otherwise, $t\sigma\not\in\NS(\syz\F)$,
so the normal form has a leading term which is strictly smaller than
$t\sigma$ and we have $\S(tf)<t\S(f)$. \end{proof}
\begin{cor}
\label{cor: tS(f) =00003D S(tf) if tS(f) =00003D S(g)}To decide whether
$\S(tf)=t\S(f)$, it suffices to know $\NS(\syz\F)$ or, equivalently,
$\LT(\syz\F)$. Also, if $t\S(f)=\S(g)$ for some $g\in I\setminus\{0\}$,
then since $\S(g)\in\NS(\syz\F)$, we can conclude that $t\S(f)=\S(tf)$.
\end{cor}
One of the key concepts of the classic theory of Gröbner bases is
the polynomial reduction: one says that $f\in P\setminus\{0\}$ reduces
with a $h\in P\setminus\{0\}$, if there exist $\alpha\in k^{*}$
and $t\in\T^{n}$ such that $\LT(f-\alpha th)<\LT(f)$, denoted 
\[
f\xrightarrow{h}g
\]
where $g=f-\alpha th$.

We now introduce a special kind of reduction for a polynomial $f$,
which takes in consideration the value of $\S(f)$.
\begin{defn}
[$\S$-reduction]Let $f,h\in I\setminus\{0\}$, $g\in I$ and $\sigma\in\Tnm$.
We say that $f$ \emph{$\S$-reduces} with respect to $\sigma$ to
$g$ with $h$, 
\[
f\xrightarrow{h}_{\S,\sigma}g
\]
 if there are $t\in\T^{n}$ and $\alpha\in k^{*}$ such that: 
\begin{itemize}
\item $\LT(g)<\LT(f)$ and $f-\alpha th=g$, and
 
\item $\S(th)<\sigma$. 
\end{itemize}
When we omit to specify $\sigma$, we assume $\sigma=\S(f)$. 
\end{defn}
Note that this reduction is defined only for polynomials $f$ which
belong to the ideal $I$, and not for abitrary elements of the ring
$P$. Also, when $\sigma=\S(f)$, since $\S(th)<\S(f)$ we have $\S(g)=\S(f-\alpha th)=\S(f)$.
Hence, when performing one, or more, $\S$-reduction steps with a
polynomial: 
\[
f\xrightarrow{h_{0}}_{\S}f_{1}\xrightarrow{h_{1}}_{\S}f_{2}\xrightarrow{h_{2}}_{\S}\dots\xrightarrow{h_{k-1}}_{\S}f_{k}
\]
 we have $\S(f)=\S(f_{1})=\ldots=\S(f_{k})$ and $\LT(f)>\LT(f_{1})>\ldots>\LT(f_{k})$;
that is, the value of $\S$ is kept constant, while the leading term
decreases.

Let us consider how to characterize those elements which cannot be
further $\S$-reduced with respect to a given $\sigma\in\Tnm$. The
following definition is natural:
\begin{defn}
[$\S$-irreducible polynomial]We say that $f\in I$ is \emph{$\S$-ir\-re\-du\-cible}
with respect to $\sigma\in\Tnm$ if $f=0$ or if there is no $h\in I$
which $\S$-reduces $f$ with respect to $\sigma$. As before, if
we do not specify $\sigma$, we assume $\sigma=\S(f)$. Note that
this definition depends on the values of $I$, $\F$ and $\mu'$.
\end{defn}
We could look for a criterion which decides whether a given set of
nonzero polynomials $G$ is a Gröbner basis by looking at the values
of $\S(g)$ for all $g$ in $G$. However, it is wiser to characterize
a set of polynomials with a property similar to that of a Gröbner
basis, but which also accounts for $\S$. We therefore introduce the
following:
\begin{defn}
[$\S$-Gröbner basis]\label{def: S-GB}We say that $G\subset I$ is
an \emph{$\S$-Gröbner basis} if for each $\S$-irreducible polynomial
$f\in I\setminus\{0\}$, there exist $g\in G$ and $t\in\T^{n}$ such
that $\LT(tg)=\LT(f)$ and $\S(tg)=\S(f)$. \end{defn}
\begin{rem}
An $\S$-Gröbner basis depends on:
\begin{itemize}
\item the ideal $I$, 
\item the term ordering $\mu$ on $\T^{n}$, 
\item the $m$-tuple of generators $\F$, 
\item the ordering $\mu'$ on $\T_{m}^{n}$. 
\end{itemize}
\end{rem}
We will prove in the following section that an $\S$-Gröbner basis
is a Gröbner basis in the usual sense. 
 While Definition~\ref{def: S-GB} is not especially useful from
a computational point of view, inasmuch as it is quantified over an
infinite set, Theorem~\ref{F5crit} will provide us an equivalent
criterion that is quantified over a finite set. Before we can prove
it, however, we need to consider some properties of $\S$-reductions.

\section{\label{sec:Properties-of-S-reductions}Properties of $\S$-reductions}

In this section we will prove the main facts which will lead to the
characterization we are looking for.
\begin{defn}
Let
\[
\begin{array}{lccl}
\varphi: & \LT(I) & \to & \NS(\syz\F)\\
 & t & \mapsto & \min\{\S(f)\mid f\in I,\ \LT(f)=t\}.
\end{array}
\]
 
\end{defn}
In other words, if $t$ belongs to $\LT(I)$, $\varphi$ is the minimum
value $\S$ can take on a polynomial whose leading term is $t$. It
follows that, for any $f\in I\setminus\{0\}$, $\varphi(\LT(f))\leq\S(f)$
always holds.
\begin{lem}
\label{lemma_varphi_bij} $\varphi$ is a bijection, and the inverse
function of $\varphi$ has an explicit formula: $\varphi^{-1}(\sigma)=\min\{t'\in\T^{n}\mid\exists f\in I,\LT(f)=t',\S(f)=\sigma\}.$ \end{lem}
\begin{proof}
We show that $\varphi$ is both injective and surjective.
\begin{description}
\item [{Injective:}] By way of contradiction, suppose there exist $\sigma\in\NS(\syz\F)$
and $t_{1},t_{2}\in\LT(I)$ such that $t_{1}>t_{2}$ and $\sigma=\varphi(t_{1})=\varphi(t_{2})$.
Then we can find $f_{1},f_{2}\in I$ such that $\LT(f_{1})=t_{1}$,
$\LT(f_{2})=t_{2}$, and $\S(f_{1})=\S(f_{2})=\sigma$. By Lemma \ref{lemma_prop_S},
there exist $\alpha,\beta\in k^{*}$ such that $\S(\alpha f_{1}+\beta f_{2})<\sigma$,
but $\LT(\alpha f_{1}+\beta f_{2})=t_{1}$, and therefore $\varphi(t_{1})<\sigma$,
contradicting the hypothesis.
\item [{Surjective:}] Let $\sigma=\tau e_{i}\in\NS(\syz\F)$, define 
\[
t=\min\{t'\in\T^{n}\mid\exists f\in I,\LT(f)=t',\S(f)=\sigma\}.
\]
(This set is not empty because it contains $\LT(\tau f_{i})$.) Let
$f\in I$ be a polynomial with $\LT(f)=t$ and $\S(f)=\sigma$; obviously
$\varphi(t)\leq\sigma$. By way of contradiction, suppose that $\varphi(t)<\sigma$.
Then there exists $f'\in I$ such that $\LT(f')=t$ and $\S(f')<\sigma$.
We can now choose $\alpha,\alpha'\in k^{*}$ with $\LT(\alpha f+\alpha'f')<t$
such that $\S(\alpha f+\alpha'f')=\sigma$. The existence of $\alpha f+\alpha'f'$
contradicts the minimality of $t$; therefore, $\varphi(t)=\sigma$. 
\end{description}
\end{proof}
The fact that $\varphi$ is a bijection will play a crucial role in
most of the subsequent proofs.
\begin{thm}
[$S$-reduction theorem]\label{Sred_thm} Let $f\in I$ and $\sigma\in\Tnm$
such that $f$ is $\S$-irreducible with respect to $\sigma$ and
$f$ is of the form $f=v(\alpha\sigma+\textrm{smaller terms})$, for
some $\alpha\in k^{*}$.

Either the following equivalent propositions hold: 

\begin{enumerate}[(a)]\item $f=0$,\label{Sred_thm_a}

\item$\sigma\in\LT(\syz\F)$, \label{Sred_thm_b}\end{enumerate}

or the following equivalent propositions hold: 
\begin{enumerate}
\item \label{Sred_thm_1} $f\neq0$, 
\item \label{Sred_thm_2} $\sigma\in\NS(\syz\F)$, 
\item \label{Sred_thm_3} $f\neq0$ and $\sigma=\S(f)=\varphi(\LT(f))$. 
\end{enumerate}
\end{thm}
\begin{proof}
~
\begin{description}
\item [{\ref{Sred_thm_a}~$\Rightarrow$~\ref{Sred_thm_b}}] Suppose
$f=0$, then $0=f=v(\alpha\sigma+\cdots)$. It follows that $\sigma\in\LT(\syz\F)$.
\item [{\ref{Sred_thm_b}~$\Rightarrow$~\ref{Sred_thm_a}}] Assume by
way of contradiction that $\sigma\in\LT(\syz\F)$ and $f\neq0$. Let
$t=\LT(f)$, and consider $\sigma'=\varphi(t)\in\NS(\syz\F)$. There
exists $g\in I$ such that $\LT(g)=t$ and $\S(g)=\sigma'$; since
$\sigma'<\sigma$, $g$ is an $\S$-reductor for $f$ with respect
to $\sigma$, contradicting the fact that $f$ is $\S$-irreducible.
Therefore, $f=0$.
\item [{\ref{Sred_thm_1}~$\Rightarrow$~\ref{Sred_thm_2}}] Assume by
way of contradiction that $f\neq0$ and $\sigma\not\in\NS\left(\syz\F\right)$.
Then $\sigma\in\LT\left(\syz\F\right)$. Let $t=\LT\left(f\right)$,
and consider $\sigma'=\varphi\left(t\right)\in\NS\left(\syz\F\right)$.
There exists $g\in I$ such that $\LT\left(g\right)=t$ and $\S\left(g\right)=\sigma'$;
since $\sigma'<\sigma$, $g$ is an $\S$-reductor for $f$ with respect
to $\sigma$, constradicting the hypothesis that $f$ is $\S$-irreducible.
Therefore, $\sigma\in\NS\left(\syz\F\right)$.
\item [{\ref{Sred_thm_2}~$\Rightarrow$~\ref{Sred_thm_3}}] Assume $\sigma\in\NS(\syz\F)$.
Necessarily, $\sigma=\S(f)$. Suppose now that $\sigma\neq\varphi(\LT(f))$;
then $\S(f)>\varphi(\LT(f))$. Therefore there exists a polynomial
$g\in I$ such that $t=\LT(g)=\LT(f)$ and $\varphi(t)=\S(g)=\varphi(\LT(f))<\S(f)$.
It follows that $g$ is an $\S$-reductor of $f$, and $f$ is not
$\S$-irreducible. 
\item [{\ref{Sred_thm_3}~$\Rightarrow$~\ref{Sred_thm_1}}] Obvious.
\end{description}
\end{proof}
Theorem~\ref{Sred_thm} implies that it only makes sense to consider
those polynomials $f$ that are $\S$-irreducible with respect to
$\S(f)$. Also, an $\S$-reduction yields $0$ if and only if performed
with respect to a $\sigma\in\LT(\syz\F)$; conversely, if an $\S$-reduction
yields a non-zero polynomial, then we know that it was performed with
respect to some $\sigma\in\NS(\syz\F)$. 

\begin{rem}
\label{rem: f S-irred iff S(f)=00003Dphi(LT(f))}Observe that a polynomial
$f$ is $\S$-irreducible iff $\S\left(f\right)=\varphi\left(\LT\left(f\right)\right)$;
otherwise, $\S\left(f\right)>\varphi\left(\LT\left(f\right)\right)$,
and we could find $g\in I$ such that $\LT\left(g\right)=\LT\left(f\right)$
and $\S\left(g\right)=\varphi\left(\LT\left(f\right)\right)$, so
that $g$ would $\S$-reduce $f$.
\end{rem}
In strict analogy with the classic Gröbner basis theory we have the
following result:
\begin{prop}
\label{SGB_Sreductor} If $G$ is an $\S$-Gröbner basis then for
any nonzero $f\in I$ such that $f$ is not $\S$-irreducible, there
exists $g\in G$ and $t\in\T^{n}$ such that: 
\begin{itemize}
\item $\LT(tg)=\LT(f)$, 
\item $\S(tg)=t\S(g)<\S(f)$. 
\end{itemize}
That is, it is always possible to find an $\S$-reductor for $f$
in $G$.\end{prop}
\begin{proof}
Since $f$ is not $\S$-irreducible, take $h$ $\S$-irreducible such
that $\LT\left(f\right)=\LT\left(h\right)$. From the remark above,
$\S\left(h\right)<\S\left(f\right)$, so $h$ is an $\S$-reductor
of $f$. We can then find $t\in\T$ and $g\in G$ such that $t\LT\left(g\right)=\LT\left(h\right)=\LT\left(f\right)$
and (using Corollary~\ref{cor: tS(f) =00003D S(tf) if tS(f) =00003D S(g)})
$\S\left(tg\right)=t\S\left(g\right)=\S\left(h\right)$. 
\end{proof}
This fact combined with lemma \ref{lemma_varphi_bij} leads immediately
to:
\begin{prop}
If $G$ is an $\S$-Gröbner basis, then $G$ is a Gröbner basis with
respect to the ordering $\mu$ on $\T^{n}$. \end{prop}
\begin{proof}
For any $t\in\LT(I)$, Lemma \ref{lemma_varphi_bij} implies that
there exists $\sigma\in\Tnm$ such that $\varphi^{-1}(\sigma)=t$.
Let $f\in I$ such that $\LT(f)=t$ and $\S(f)=\sigma$. From Proposition~\ref{SGB_Sreductor},
we may assume that $f$ is $\S$-irreducible (if not, $\S$-reduce
it). Then $\exists g\in G$, $u\in\T^{n}$ such that $\LT(ug)=\LT(f)=t$.
Hence the set $\{\LT(g)\mid g\in G\}$ generates $\LT(I)$ and $(G)\subseteq I$.
Therefore $G$ is a Gröbner basis for $I$.\end{proof}
\begin{prop}
\label{prop_finiteSGb} Every $\S$-Gröbner basis contains a finite
$\S$-Gröbner basis. 
\end{prop}
This proof's reference to ``monomodule'' is not a misspelling; see~\cite{CompCommAlgI}
for more information.
\begin{proof}
Let $G=\{g_{i}\}_{i\in\mathcal{I}}$ be an $\S$-Gröbner basis. Define
the map 
\[
\begin{array}{lccl}
\vartheta: & G & \to & \T^{n}\oplus\T_{m}^{n}\\
 & g_{i} & \mapsto & (\LT(g_{i}),\S(g_{i})).
\end{array}
\]
The image $\vartheta(G)$ generates a submodule $M$ of the $\T^{n}\oplus\T_{m}^{n}$-monomodule
$\T^{n}\oplus\T_{m}^{n}$. This is also a noetherian monomodule;
therefore, there exists a finite subset $\mathcal{J}$ of $\mathcal{I}$
such that $\vartheta(G')$ generates $M$, for some $G'=\{g_{j}\}_{j\in\mathcal{J}}$.

We claim that $G'$ is itself an $\S$-Gröbner basis. To see this,
let $f\in I$ be an $\S$-irreducible polynomial. By definition, $\S\left(f\right)\in\NS\left(\syz\F\right)$.
Since $G$ is an $\S$-Gröbner basis, we can find a $g_{i}\in G$
and a $t\in\T^{n}$ such that $t\S\left(g_{i}\right)=\S(tg_{i})=\S(f)$
and $t\LT\left(g_{i}\right)=\LT(tg_{i})=\LT(f)$ (using Lemma~\ref{lemma_prop_S}(\ref{lemma_prop_S3})
for $t\S\left(g_{i}\right)=\S\left(tg_{i}\right)$). If $i\in\mathcal{J}$,
then $g_{i}\in G'$ and we're fine. Otherwise, $i\in\mathcal{I}\setminus\mathcal{J}$;
since $\vartheta(g_{i})\in M$, there exist $j_{i}\in\mathcal{J}$,
$u\in\T^{n}$, and $ve_{k}\in\T_{m}^{n}$ such that 
\[
\left(u,ve_{k}\right)\cdot\vartheta(g_{j_{i}})=\vartheta(g_{i}).
\]
We consider three cases.

If $u=v$, then $t'=ut\in\T^{n}$ satisfies $t'\LT\left(g_{j_{i}}\right)=\LT\left(f\right)$
and $t'\S\left(g_{j_{i}}\right)=\S\left(f\right)$, so we're fine.

If $u<v$, then $t'=ut\in\T^{n}$ satisfies $t'\LT\left(g_{j_{i}}\right)=\LT\left(f\right)$
and $t'\S\left(g_{j_{i}}\right)<\S\left(f\right)$, contradicting
the hypothesis that $f$ is $\S$-irreducible.

If $u>v$, then there exist $\alpha\in k$ and $t'=vt\in\T^{n}$ such
that $p=f-\alpha t'g_{j_{i}}$ satisfies $\LT\left(p\right)=\LT\left(f\right)$,
but $\S\left(p\right)<\S\left(f\right)$, contradicting the hypothesis
that $f$ is $\S$-irreducible.

Since the other two cases lead to contradiction, we have found $g_{j_{i}}\in G'$
and $t'\in\T^{n}$ which satisfy the $\S$-Gröbner basis property
for $f$. Since $f$ was an arbitrary $\S$-irreducible element of
$I$, we conclude that $G'$ is an $\S$-Gr\"obner basis.
\end{proof}
The elements of $G'$ will prove critically important when we examine
our algorithm, so we will identify them by a special term.
\begin{defn}
[Primitive $S$-irreducible polynomial]We say that a nonzero polynomial
$f$ $\S$-irreducible with respect to $\S(f)$ is \emph{primitive}
\emph{$\S$-irreducible }if there are no polynomials $f'\in I\setminus\{0\}$
and terms $t\in\T^{n}\backslash\left\{ 1\right\} $ such that $f'$
is $\S$-irreducible, $\LT(tf')=\LT(f)$ and $\S(tf')=\S(f)$.
\end{defn}
The proof of Proposition~\ref{prop_finiteSGb} implies that if we
have an $\S$-Gröbner basis $G$, then we can obtain a finite $\S$-Gröbner
basis by keeping a subset of primitive $\S$-irreducible polynomials
with different leading terms. Hence there exist $\S$-Gröbner bases
which contain only primitive $\S$-irreducible polynomials.

\section{\label{sec:The main result}The main result}

First we adapt the definition of a normal pair in~\cite{f5} to reflect
primitive $\S$-irreducible polynomials.
\begin{defn}
[Normal Pair]\label{def_normpair} Given $g_{1},g_{2}\in I\setminus\{0\}$,
let $\Spol(g_{1},g_{2})=u_{1}g_{1}-u_{2}g_{2}$ be the S-polynomial
of $g_{1}$ and $g_{2}$; that is, $u_{i}=\frac{\lcm(\LT(g_{1}),\LT(g_{2}))}{\LC(g_{i})\LT(g_{i})}$.
We say that $(g_{1},g_{2})$ is a \emph{normal pair} if: 
\begin{enumerate}
\item \label{def_normpair_1} $g_{i}$ is a primitive $\S$-irreducible
polynomial for $i=1,2$, 
\item \label{def_normpair_2} $\S(u_{i}g_{i})=\LT(u_{i})\S(g_{i})$ for
$i=1,2$, 
\item \label{def_normpair_3} $\S(u_{1}g_{1})\neq\S(u_{2}g_{2})$. 
\end{enumerate}
\end{defn}
\begin{rem}
\label{rem_u1neq1} With this definition, if $(g_{1},g_{2})$ is a
normal pair, then 
\[
\S(\Spol(g_{1},g_{2}))=\max(\S(u_{1}g_{1}),\S(u_{2}g_{2}))
\]
 will always hold. In addition, if $\S(u_{1}g_{1})>\S(u_{2}g_{2})$,
then $u_{1}\neq1$, as if $u_{1}$ were $1$, $g_{2}$ would be an
$\S$-reductor of $g_{1}$. Therefore $\S(\Spol(g_{1},g_{2}))>\max(\S(g_{1}),\S(g_{2}))$. \end{rem}
\begin{thm}
[F5 criterion]\label{F5crit} Suppose that $G$ is a set of $\S$-irreducible
polynomials of $I$, such that: 
\begin{itemize}
\item for each $i=1,\ldots,m$ such that $e_{i}\not\in\LT(\syz\F)$ there
exists $g_{i}\in G$ such that $\S(g_{i})=e_{i}$, and
\item for any $g_{1},g_{2}\in G$ such that $(g_{1},g_{2})$ is a normal
pair, there exist $g\in G$ and $t\in\T^{n}$ such that $tg$ is $\S$-irreducible
and $\S(tg)=\S(\Spol(g_{1},g_{2}))$. 
\end{itemize}
Then $G$ is a $\S$-Gröbner basis of $I$. \end{thm}
\begin{rem}
[Rewritable criterion]\label{rem_chooseS} Note that the second condition
does not explicitly involve the S-po\-ly\-no\-mial of a pair $(g_{1},g_{2})$,
but cares only about $\S(\Spol(g_{1},g_{2}))$. Hence, we can think
of this as a criterion to choose elements of $\NS(\syz\F)$ instead
of polynomials. Additionally, if two or more normal pairs are such
that $\S$ takes the same value on their S-polynomials, we can freely
consider just one of them.\end{rem}
\begin{proof}
As noted at the end of the previous section, we may, without loss
of generality, assume that the elements of $G$ are primitive $\S$-irreducible
and have distinct leading terms. By way of contradiction, suppose
that there exists a minimal $\sigma\in\NS(\syz\F)$ and an $\S$-irreducible
$f\in I\backslash\left\{ 0\right\} $ with $\S\left(f\right)=\sigma$
and the $\S$-Gröbner basis property does not hold for $f$ and $\sigma$.
That is, for all $g\in G$ and for all $t\in\T^{n}$, $\LT\left(tg\right)\neq\LT\left(f\right)$
or $\S\left(tg\right)\neq\S\left(f\right)$.

The first hypothesis implies that there exist at least one primitive
$\S$-irreducible $g\in G$ and some $\tau\in\T^{n}$ such that $\tau\S(g)=\S\left(f\right)=\sigma$;
among the possible choices for $g$ and $\tau$, pick one which minimizes
$\LT(\tau g)$. By Lemma~\ref{lemma_prop_S}(\ref{lemma_prop_S3}),
$\S\left(\tau g\right)=\tau\S\left(g\right)=\sigma$. Hence $\LT\left(\tau g\right)\neq\LT\left(f\right)$.
By Remark~\ref{rem: f S-irred iff S(f)=00003Dphi(LT(f))}, $\S\left(f\right)=\varphi\left(\LT\left(f\right)\right)$,
and by Lemma~\ref{lemma_varphi_bij}, $\LT\left(\tau g\right)>\LT\left(f\right)$.
In addition, we have $\S\left(\tau g\right)=\S\left(f\right)=\varphi\left(\LT\left(f\right)\right)\neq\varphi\left(\LT\left(\tau g\right)\right)$,
so again by Remark~\ref{rem: f S-irred iff S(f)=00003Dphi(LT(f))},
$\tau g$ is not $\S$-irreducible.

By Lemma \ref{lemma_prop_S}(\ref{lemma_prop_S2}), there exist $\alpha,\beta\in k^{*}$
such that $\S(\alpha f+\beta\tau g)=\sigma'$ for some $\sigma'<\sigma$.
Since $\sigma$ was chosen to be the minimal element of $\NS\left(\syz\F\right)$
such that the $\S$-Gr\"obner basis property does not hold, Definition~\ref{def: S-GB}
and Proposition~\ref{SGB_Sreductor} applied to $\alpha f+\beta\tau g$
imply that there exist $g'\in G$ and $\tau'\in\T^{n}$ such that
$\LT\left(\tau'g'\right)=\LT\left(\alpha f+\beta\tau g\right)=\LT\left(\tau g\right)$
and 
\[
\S\left(\tau'g'\right)=\tau'\S\left(g'\right)\leq\S\left(\alpha f+\beta\tau g\right)=\sigma'<\sigma=\S\left(\tau g\right).
\]
Clearly $g\neq g'$.

It follows that $(g,g')$ is a normal pair. From the second hypothesis,
we know that there exist $g''\in G$ and $\tau''\in\T^{n}$ such that
$\tau''g''$ is $\S$-irreducible and $\S(\tau''g'')=\S(\Spol(g,g'))$.
Write $\hat{\tau}\Spol(g,g')=\gamma\tau g-\gamma'\tau'g'$, for some
$\gamma,\gamma'\in k^{*}$, where $\hat{\tau}$ is the gcd of $\tau$
and $\tau'$. Since $\left(g,g'\right)$ is a normal pair and $\sigma\in\NS\left(\syz\F\right)$,
\[
\sigma=\tau\S\left(g\right)=\widehat{\tau}\S\left(\Spol\left(g,g'\right)\right)=\widehat{\tau}\S\left(\tau''g''\right)=\S\left(\widehat{\tau}\tau''g''\right).
\]
By Remark~\ref{rem: f S-irred iff S(f)=00003Dphi(LT(f))}, $\S\left(\tau''g''\right)=\varphi\left(\LT\left(\tau''g''\right)\right)$,
so we have $\LT(\tau''g'')=\varphi^{-1}\left(\S\left(\tau''g''\right)\right)\leq\LT(\Spol(g,g'))$.
Multiplying both sides by $\hat{\tau}$, we have 
\[
\LT(\hat{\tau}\tau''g'')\leq\LT(\gamma\tau g-\gamma'\tau'g')<\LT(\tau g).
\]
The existence of $g''$ and $\hat{\tau}\tau''$ contradicts the choice
of $g$ and $\tau$.
\end{proof}

\section{\label{sec:The-algorithm}The algorithm}

We shall now present a simple algorithm which computes as $\S$-Gröbner
basis of an ideal based on the criterion. This algorithm is quite
different from Faugère's, in that it is a direct application of the
criterion. In particular, it does not involve reductions that yield
more then one result, nor the more rigorous simplification rules.
See Section~\ref{sub:Comparison-with-F5} for a detailed discussion.

One first problem is that to check condition \ref{def_normpair_2}
of definition \ref{def_normpair} we need to know $\LT(\syz\F)$,
since 
\[
\S(tf)=t\S(f)\iff t\S(f)\not\in\LT(\syz\F).
\]
We almost never know this before hand; therefore, we introduce a new
variable $L$, a subset of $\LT(\syz\F)$. At the beginning of the
algorithm, we simply assume $L=\varnothing$. We make use of $L$
whenever we need to check if $t\S(f)=\S(tf)$ by checking whether
$t\S(f)$ belongs to $\langle L\rangle\subseteq P^{m}$, the $P$-module
generated by $L$. We then replace condition \ref{def_normpair_2}
of definition \ref{def_normpair} by: 
\[
\S(u_{i}g_{i})=\LT(u_{i})\S(g_{i})\iff\LT(u_{i})\S(g_{i})\not\in\langle L\rangle.
\]
By doing so, we end up considering more pairs than we should, but
we do not skip any legitimate pair.

So, when $(g_{1},g_{2})$ is a normal pair (with the weakened condition
\ref{def_normpair_2}), we calculate a polynomial $f=\Spol(g_{1},g_{2})$
and a $\sigma=\max(u_{1}g_{1},u_{2}g_{2})$. Thereafter we $\S$-reduce
$f$ with respect to $\sigma$. Note that $\sigma$ satisfies the
hypothesis of Theorem \ref{Sred_thm}. If the $\S$-reduction yields
$0$, we know that $\sigma\in\LT(\syz\F)$; accordingly, we enlarge
$L$ by inserting $\sigma$. Otherwise, we obtain a nonzero polynomial,
which tells us that $\sigma=\S(f)$. 

$G$ is the set which will contain the $\S$-Gröbner basis; we add
elements to $G$ as we find them. For each element $g$ we add to
$G$, we also store $\S(g)$; thus, $G$ is more precisely a set of
pairs $(g,\sigma)$. When an $\S$-reduction returns a nonzero polynomial
$f$, we insert $f$ into $G$. Initially, $G=\emptyset$, rather
than a set containing $\{f_{i}\}$, since we do not know if $f_{i}$
is $\S$-irreducible.

$B$ is the set of pairs of the form $(f,\sigma)$, where $f$ is
a polynomial that we $\S$-reduce with respect to $\sigma$. Initially,
we know that $\S(f_{i})=e_{i}$; therefore, we initialize $B$ as
$\{(f_{1},e_{1}),\ldots,(f_{m},e_{m})\}$.

The idea of the algorithm is to build an $\S$-Gröbner basis by finding
its elements in ascending value of $\S$; that is, always to choose
$\left(f,\sigma\right)\in B$ such that $\sigma$ is minimal. (See
step \ref{algostep_minB}.)

\begin{rem}
\label{rem: rewritten_in_algorithm}In virtue of remark~\ref{rem_chooseS},
for each $\sigma$ we can keep in $B$ at most one polynomial $f$
such that $\S\left(f\right)=\sigma$. For the same reason we can,
at any time, remove $\left(f,\sigma\right)$ from $B$ if we can find
another polynomial $f'$ such that $\S\left(f'\right)=\sigma$ and
$\LT\left(f'\right)<\LT\left(f\right)$.

In practice we will remove from $B$ a pair $\left(f,\sigma\right)$
if we can find a $t\in\T^{n}$ and a $\left(f',\sigma'\right)\in G\cup B$
such that $t\sigma'=\sigma$ and $t\LT\left(f'\right)<\LT\left(f\right)$.
\end{rem}
The pseudo code of the algorithm is the following:
\begin{lyxalgorithm}
\label{algo_SGb} \mbox{} 
\begin{description}
\item [{Input:}] $\F=\left(f_{1},\ldots,f_{m}\right)$: an element of $P^{m}$,

$\mu$: an ordering on $\T^{n}$,

$\mu'$: an ordering on $\Tnm$.

\item [{Output:}] $G$: an $\S$-Gröbner basis of $I=\left(f_{1},\ldots,f_{m}\right)$. \end{description}
\begin{enumerate}
\item $L:=\varnothing$ 
\item $G:=\varnothing$ 
\item $B:=\left\{ \left(f_{1},e_{1}\right),\ldots,\left(f_{m},e_{m}\right)\right\} $ 
\item \label{algostep_loopB}While $B\neq\varnothing$ 

\begin{enumerate}
\item \label{algostep_refineB}$B:=\left\{ \left(f,\sigma\right)\in B\mid\sigma\not\in\langle L\rangle\right\} $
\item \label{algostep_refineB_rewritten}Remove from $B$ any $\left(f,\sigma\right)$
such that we can find $\left(f',\sigma'\right)\in G\cup B$, $t\in\T^{n}$
satisfying $t\sigma'=\sigma$ and $\LT\left(tf'\right)<\LT\left(f\right)$
\item \label{algostep_minB}Pick $\left(f,\sigma\right)\in B$ with minimal
$\sigma$.
\item \label{algostep_reduce}$f:=$ $\S$-reduce$\left(f,\sigma,G\right)$
\item If $f\neq0$ then

\begin{enumerate}
\item \label{algostep_updatepairs}$B:=$ UpdatePairs$\left(L,G,B,\left(f,\sigma\right)\right)$
\item \label{algostep_updateG}$G:=G\cup\left\{ \left(f,\sigma\right)\right\} $
\end{enumerate}
\item Else

\begin{enumerate}
\item \label{algostep_LuS}$L:=L\cup\left\{ \sigma\right\} $
\end{enumerate}
\end{enumerate}
\item Return $\left\{ g:\left(g,\sigma\right)\in G\right\} $
\end{enumerate}
\end{lyxalgorithm}
Note that, since $L$ may change during each iteration, some pairs
we assumed to be normal turn out not to be normal. We remove those
in step \ref{algostep_refineB}.

In step \ref{algostep_refineB_rewritten} we implement the idea presented
in Remark \ref{rem: rewritten_in_algorithm}. Note that this is an
optimization; the algorithm will successfully terminate without this
line.

We still have to describe the two procedures Algorithm \ref{algo_SGb}
invokes. The first is \emph{$\S$-reduce}:
\begin{lyxalgorithm}
[$\S$-reduce]\label{algo_Sred}~
\begin{description}
\item [{Input:}] $f$: an element of $I$,

$\sigma$: an element of $\Tnm$,

$G$: a set that contains the elements $\left(g,\S\left(g\right)\right)$
of an $\S$-Gröbner basis with $\S(g)<\sigma$. 

\item [{Output:}] $f$: an $\S$-irreducible polynomial with respect to
$\sigma$.\end{description}
\begin{enumerate}
\item $f:=f/\LC\left(f\right)$
\item While $\exists\left(g,\S\left(g\right)\right)\in G,t\in\T^{n}$ such
that $t\LT\left(g\right)=\LT\left(f\right)$ and $t\S\left(g\right)<\sigma$ 

\begin{enumerate}
\item $f:=f-tg/\LC\left(g\right))$
\item If $f=0$ then Return $0$ 
\item $f:=f/\LC\left(f\right)$
\end{enumerate}
\item Return $f$
\end{enumerate}
\end{lyxalgorithm}
This algorithm takes as input a polynomial $f$ and a $\sigma\in\Tnm$
and, as long as there is an $\S$-reductor for $f$ in $G$, performs
$\S$-reduction steps. Because of the hypothesis on $G$ we know we
obtain an $\S$-irreducible polynomial with respect to $\sigma$.

The second is \emph{UpdatePairs}:
\begin{lyxalgorithm}
[UpdatePairs]\label{algo_UpdatePairs}~
\begin{description}
\item [{Input:}] $L$: a subset of $\LT\left(\syz\F\right)$,

$G$: a set that contains the elements $\left(g,\S\left(g\right)\right)$
of a $\S$-Gröbner basis with $\S\left(g\right)<\sigma$,

$B$: a set that contains elements $\left(g,\sigma_{g}\right)$ of
polynomials that have yet to be considered,

$\left(f,\S\left(F\right)\right)$: where $f\in I\backslash\left\{ 0\right\} $.

\item [{Output:}] $B'$: a set of pairs $\left(f',\sigma\right)$ that
satisfy Theorem \ref{Sred_thm}, produced by the criterion.\end{description}
\begin{enumerate}
\item $B':=\varnothing$
\item For each $\left(g,\S\left(g\right)\right)\in G$, if $\left(f,g\right)$
is a normal pair

\begin{enumerate}
\item Compute $u_{1}$, $u_{2}$ such that $\Spol\left(f,g\right)=u_{1}f+u_{2}g$
\item $\sigma:=\max\left(u_{1}\S\left(f\right),u_{2}\S\left(g\right)\right)$
\item \label{algostep_updateBprime}$B':=B'\cup\{(\Spol(f,g),\sigma)\}$
\end{enumerate}
\item Return $B'\cup B$
\end{enumerate}
\end{lyxalgorithm}
\begin{prop}
Algorithm \ref{algo_SGb} terminates. \end{prop}
\begin{proof}
First we show that step \ref{algostep_LuS} is executed only a finite
number of times.

Because of step \ref{algostep_refineB}, at a given time, we only
consider $\sigma$ that do not belong to $L$; so when we execute
step \ref{algostep_LuS} we really enlarge the $P$-module generated
by $L$. Since $P^{m}$ is noetherian this can happen only a finite
number of times.

Also, that step \ref{algostep_updatepairs} is executed only a finite
number of times. First note that if $f$ is not primitive $\S$-irreducible
(that is, $f$ is only $\S$-irreducible), then Algorithm~\ref{algo_UpdatePairs}
does nothing, so no new polynomials are generated. In the proof of
Proposition~\ref{prop_finiteSGb}, we see that an $\S$-Gröbner basis
contains only a finite number of primitive $\S$-irreducible polynomials.
This completes the proof.\end{proof}
\begin{thm}
Algorithm \ref{algo_SGb} computes an $\S$-Gröbner basis of $I$. \end{thm}
\begin{proof}
This is a direct consequence of criterion of Theorem~\ref{F5crit}:
Previous remarks have shown that $G$ contains only $\S$-irreducible
polynomials, and the initial value of $B$ ensures that the algorithm
satisfies the first condition. For the second condition, for each
normal pair $(g_{1},g_{2})$, we ensure that we have a polynomial
$f$ and a monomial $t$ such that $\S(tf)=\S(\Spol(g_{1},g_{2}))$.
\end{proof}
The fact that non-primitive $\S$-irreducible polynomials do not generate
any new pairs plays a \emph{central} role in this proof of termination.
Without it, the thesis does not hold: if we drop condition~\eqref{def_normpair_1}
of Definition~\ref{def_normpair}, it is possible that the algorithm
could enter an infinite loop, computing an infinite number of polynomials
of the form $t_{i}f$ where $\{t_{i}\}$ is an infinite set of terms
and $f$ is an $\S$-irreducible polynomial and each of the $t_{i}f$
is $\S$-irreducible itself. (This occurs, for example, in the implementation
of~\cite{cryptoeprint:2006:404}.)

\section{\label{sec:Comparison-with-previous}Comparison with previous work}

In this section, we consider how this algorithm is both similar and
different to two algorithms in past work: the staggered linear basis
algorithm of~\cite{StaggeredLinearBases} and the F5 algorithm of~\cite{f5}.
(Another discussion of the relationship between F5 and the staggered
linear basis algorithm can be found in~\cite{MoraPosso}.) We also
illustrate explicit differences on three particular examples.

\subsection{\label{sub:Comparison-with-SLB}Comparison with Staggered Linear
Bases}

The Staggered Linear Basis algorithm (in the rest of this section,
SLB)~\cite{StaggeredLinearBases,MoraPosso} introduced a special
kind of Gr\"obner basis.
\begin{defn}
The set $B\subset I$ is a \emph{staggered linear basis of the ideal
$I$ if for all $f\in P$}
\begin{itemize}
\item if $f,g\in B$ and $\LT\left(f\right)=\LT\left(g\right)$, then $f=g$;
and
\item if $t\in\T$ and $tf\in B$, then $f\in B$.
\end{itemize}
\end{defn}
A full review of SLB is beyond the scope of this paper, but it is
worth comparing to the present algorithm because both use trivial
syzygies to detect zero reductions. To facilitate the explanation,
we temporarily adopt the notation $t_{i}=\LT\left(f_{i}\right)$ and
$t_{i,j}=\lcm\left(t_{i},t_{j}\right)$.

SLB tracks monomial ideals for each polynomial among the generators.
Initially, we have
\[
Z_{i}=\left(t_{1},t_{2},\ldots,t_{i-1}\right).
\]
Critical pairs $\left(f_{i},f_{j}\right)$ (with $i<j$) are rejected
whenever $t_{ij}/t_{j}\in Z_{i}$. If instead the $S$-poly\-nomial
of $\left(f_{i},f_{j}\right)$ is computed, then $Z_{j}$ is expanded
by adding the ideal generated by $t_{ij}/t_{j}$. If reduction of
the $S$-poly\-nomial results in a new polynomial $f_{k}$ being
added to the basis, SLB also creates a new ideal
\[
Z_{k}=\left(Z_{j}+\left(t_{i}\right)\right):\left(t_{ij}/t_{j}\right)+\left(t_{1},\ldots,t_{k-1}\right).
\]

Despite the use of principal syzygies in the initial definition of
$Z_{i}$, a fundamental difference between the algorithms lies in
the fact that SLB does not compute, let alone consider, the leading
module term $\S\left(f\right)$ of any polynomials. So a polynomial
can be $\S$-irreducible even if it is top-reducible, and the normal
pairs of the F5 Criterion are not the same as the critical pairs of
SLB. As a result, the approach in SLB behaves quite differently, and
fails to detect certain zero reductions detected by F5 and the present
algorithm.

\subsection{\label{sub:Comparison-with-F5}Comparison with F5}

At first glance, algorithm \ref{algo_SGb} may appear very different
from the F5 algorithm. However, if we define $\mu'$ to be 
\[
te_{i}<_{\mu'}se_{j}\iff\begin{cases}
i<j\quad\text{or}\\
i=j\textrm{ and }t<s
\end{cases}
\]
 for any $t,s\in\T^{n}$ and $1\leq i,j\leq m$, there is an interesting
relationship between $\S$-Gröbner bases and $\LT(\syz\F)$. Define,
for $1\leq l\leq m$, $\pi_{l}:P^{m}\to P$ as the projection on the
$l$-th component, then 
\begin{equation}
\pi_{l}(\psyz\F)=(f_{1},\ldots,f_{l-1})\label{eq_psyzLT}
\end{equation}
 where $\psyz(\F)$ is the $P$-module of principal syzygies, defined
as $\psyz(\F)=\left\langle f_{i}e_{j}-f_{j}e_{i}\right\rangle _{P}\subseteq P^{m}$
; $\psyz(\F)$ is clearly a $P$-submodule of $\syz(\F)$.

Suppose we have $f\in I\setminus\{0\}$ and we know $\S(f)=te_{i}$,
for some $t\in\T^{n}$. It follows from the definition of $\S$ that
$f\in(f_{1},\ldots,f_{i})$. Hence, \eqref{eq_psyzLT} implies that
\[
\LT(f)e_{i+j}\in\LT(\psyz(\F))\quad\textrm{for some }j\geq1.
\]
With this choice for $\mu'$, we can improve the performance of Algorithm
\ref{algo_SGb} by adding an instruction right after step \ref{algostep_updateG}:
\[
L:=L\cup\{\LT(f)e_{i+1},\LT(f)e_{i+2},\ldots,\LT(f)e_{m}\},
\]
where $\sigma=te_{i}$ for some $t\in\T^{n}$. In other words, whenever
we find a new element of $G$, we also find new elements of $L$.

Also, due to the ordering on $P^{m}$, the structure of an $\S$-Gröbner
basis $G$ is very special. We find the elements of $G$ in ascending
value of $\S$: we first find all the elements $g$ such that $\S(g)=te_{1}$
for some $t\in\T^{n}$, then those $g$ such that $\S(g)=te_{2}$
for some $t\in\T^{n}$ and so on. Is easy to see that the real value
of $f_{l}$ is never considered in any computation, until the algorithm
has finished producing all the elements of $G$ with $\S(g)=te_{i}$
for some $t\in\T^{n}$ and $i<l$. If we make the further assumption
that $\syz(\F)=\psyz(\F)$, we conclude that the algorithm never reduces
a polynomial to $0$, since we discover every leading term of the
syzygies in advance.

Therefore, we may say that, in this case, Algorithm \ref{algo_SGb}
is \emph{incremental}, as it first produces an $\S$-Gröbner basis
of $(f_{1})$, then an $\S$-Gröbner basis of $(f_{1},f_{2})$ and
so on, and avoids all the reductions to zero; this behavior is the
same as Faugère's F5 algorithm.

We can couch the use of ``simplification rules'' in F5~\cite[sect. 6]{f5},
also called the \emph{rewritable criterion}, in vocabulary similar
to that used in this paper: F5's algorithm to compute $S$-poly\-nomials
(SPol) discards any $(te_{i},f)$ when
\begin{itemize}
\item there exists some other $\left(t'e_{i},f'\right)\in G\cup B\cup B'$
such that $t\mid t'$, and
\item $f'$ was computed \emph{before} $f$.
\end{itemize}
This concept is related to Remark~\ref{rem_chooseS} in this paper;
roughly we know we can ``decide'' how to obtain an $\S$-irreducible
polynomial with a given signature. We prefer to start with a polynomial
with the smallest leading term we know of, while in F5 just the first
generated polynomial is kept.

This parallel carries over to the computation of $L$, which here
is used to prevent the computation of any $te_{i}\in\LT\left(\syz\F\right)$
more than once. When a polynomial is reduced to zero in F5, the simplification
rule is added even though the polynomial is discarded, and this rule
ensures that any polynomial $f$ with $\S\left(f\right)=t'e_{i}$,
where $t\mid t'$, is not computed. In other words, F5 has an implicit
provision for avoiding the computation of non-trivial syzygies, like
the algorithm here.

\subsection{\label{sub:Concrete-examples}Concrete examples}

We examine how all three examples perform on three ``standard''
systems:
\begin{itemize}
\item the system ``MMT92'', $F=\left\{ yz^{3}-x^{2}t^{2},xz^{2}-y^{2}t,x^{2}y-z^{2}t\right\} $
from \cite{f5} (this seems first to appear in non-homogenized form
in~\cite{143343});
\item the homogenized Cyclic-5 system; and
\item the homogenized Katsura-5 system.
\end{itemize}
We consider
\begin{enumerate}
\item the number of zero reductions; and
\item the size of the Gr\"obner basis generated.
\end{enumerate}
The tests were carried out in unoptimized implementations of each
algorithm in Sage~\cite{sage,SAGEImplementationOfF5,SAGEImplementationOfF5Arri,SAGEImplementationOfSLB},
and are available online.

The results are in Tables~\ref{tab: num zero reds} and~\ref{tab: size of basis}.
\begin{table}

\begin{centering}
\begin{tabular}{|c|c|c|c|c|}
\multicolumn{1}{c}{} & \multicolumn{4}{c}{\textbf{Number of zero reductions}}\tabularnewline
\cline{2-5} 
\multicolumn{1}{c|}{\textbf{Algorithm}} & MMT92 & Cyclic-5 & Cyclic-6 & Katsura-5\tabularnewline
\hline 
Staggered Linear Basis & 3 & 46 & 446 & 10\tabularnewline
\hline 
F5 & 0 & 0 & 16 & 0\tabularnewline
\hline 
Algorithm~\ref{algo_SGb} & 0 & 0 & 8 & 0\tabularnewline
\hline 
\end{tabular}
\par\end{centering}

\protect\caption{\label{tab: num zero reds}Number of zero reductions during execution
of algorithms SLB, F5, and Algorithm~\ref{algo_SGb}.}

\end{table}
\begin{table}
\begin{centering}
\begin{tabular}{|c|c|c|c|c|}
\multicolumn{1}{c}{} & \multicolumn{4}{c}{\textbf{Size of basis}}\tabularnewline
\cline{2-5} 
\multicolumn{1}{c|}{\textbf{Algorithm (size of red. GB)}} & MMT92 & Cyclic-5 & Cyclic-6 & Katsura-5\tabularnewline
\hline 
Staggered Linear Basis & 8 & 38 & 99 & 22\tabularnewline
\hline 
F5 & 10 & 39 & 202 & 30\tabularnewline
\hline 
Algorithm~\ref{algo_SGb} & 10 & 39 & 155 & 30\tabularnewline
\hline 
\end{tabular}
\par\end{centering}

\protect\caption{\label{tab: size of basis}Size of the Gr\"obner basis computed by
algorithms SLB, F5, and Algorithm~\ref{algo_SGb}.}
\end{table}
 Table~\ref{tab: num zero reds} shows that the Staggered Linear
Basis algorithm computes some zero reductions \emph{even though} the
systems are regular sequences. Neither F5 nor Algorithm~\ref{algo_SGb}
computes any zero reductions except in Cyclic-6, which is not a regular
sequence. In that system, Algorithm~\ref{algo_SGb} computes a smaller
basis, and it computes fewer zero reductions. This appears to be due
to the fact that it proceeds by ascending signature (line~\ref{algostep_minB})
rather than by ascending lcm (compare to algorithm Spol in~\cite{f5}).

\section{\label{sec:Conclusions-and-future}Conclusions and future work}

This paper has reformulated the F5 criterion, which in its original
form is due to \cite{f5}, and provided a new proof of this criterion's
correctness. We have introduced the ideas of \emph{$\S$-Gröbner basis}
and \emph{$\S$-irreducible polynomials,} and have shown that if a
set of polynomials $G$ satisfies the F5 criterion, then $G$ is an
$\S$-Gröbner basis and not just a Gröbner basis. In this new setting,
we were able to drop many restrictions present in \cite{f5}: we can
freely choose any ordering on $P^{m}$, and there is no need to for
the sequence $(f_{1},\ldots,f_{m})$ to be regular.

Our statement of the criterion is quite different from the original:
we require that all the polynomials in the set $G$ be $\S$-irreducible;
we require that if $e_{i}\not\in\NS(\syz\F)$, then there exist $g\in G$
such that $\S\left(g\right)=e_{i}$; and we impose a condition on
the \emph{signature} $\S\left(\Spol\left(g_{1},g_{2}\right)\right)$,
rather than the usual condition that
\[
\Spol\left(g_{1},g_{2}\right)=\sum_{i=1}^{\#G}h_{i}g_{i}\quad\mbox{such that}\quad h_{i}\neq0\;\Longrightarrow\LT\left(h_{i}\right)\LT\left(g_{i}\right)\leq\LT\left(\Spol\left(g_{1},g_{2}\right)\right).
\]
(Faug\`ere calls this latter condition $o\left(\Spol\left(g_{1},g_{2}\right)\right)$.)
We also changed the definition of \emph{normal pair} by adding a new
condition: the fact that we can consider only primitive $\S$-irreducible
polynomials.

We then proposed a simple algorithm to show an application of the
new criterion. The algorithm presented here is mainly demonstrative,
and does not include many ``obvious'' optimizations such as holding
off on the computation of a new polynomial $f$ until it is actually
needed in step~\ref{algostep_reduce} of Algorithm~\ref{algo_SGb}.
\begin{acknowledgement*}
The authors would like to thank the referees for helpful and instructive
comments that improved the paper.
\end{acknowledgement*}

\section*{Appendix}

In March 2016 we were informed that the arXiv version of this paper
differed from the published version in a critical point of the proof
of Proposition~14: the arXiv version used the term ``module'' where
the published version used ``monomodule,'' a much less familiar
term. This was an unfortunate error in the arXiv version, and has
now been fixed. See~\cite{CompCommAlgI} for the definition of monomodule
and some examples, including $\T_{m}^{n}$, which in their notation
is $\mathbb{T}^{n}\left\langle e_{1},\ldots,e_{m}\right\rangle $.

In October 2011, Vasily Galkin of Moscow State University contacted
us with a question about Proposition~\ref{prop_finiteSGb}. His question
was sparked by the definition of $\mu'$ as an \emph{admissible} ordering
on $\Tnm$. Apparently, we used the wrong word; our reference for
the notation (which, it amazes us to report now, we did not include
in the bibliography) was \cite{CompCommAlgI}. (As far as we can tell,
it is the only textbook that uses the word ``monomodule''.) This
text \emph{does not define} an ``admissible'' ordering for a module;
it defines either a module ordering (p.~54) or a compatible ordering
(p.~55). It seems that when we wrote ``admissible'', we meant ``compatible''.
Indeed, if the ordering is not compatible, the $\S$-Gr\"obner basis
may be infinite, as the following example shows.
\begin{example}
Let $<$ be the degrevlex ordering with $x>y>z$. Let $<'$ be the
module ordering with
\[
x\sigvar_{1}<'y\sigvar_{1}<'x\sigvar_{2}<'y\sigvar_{2}<'z\sigvar_{1}<'z\sigvar_{2}
\]
extended to all other module terms in the following way: $t\sigvar_{i}<'u\sigvar_{j}$
if
\begin{itemize}
\item $\deg t<\deg u$ (total degree), or
\item $\deg t=\deg u$ and

\begin{itemize}
\item $\deg_{z}t<\deg_{z}u$, or
\item $\deg_{z}t=\deg_{z}u$ and

\begin{itemize}
\item $i=1$ and $j=2$, or
\item $i=j$ and $\deg_{y}t<\deg_{y}u$, or
\item $i=j$, $\deg_{y}t=\deg_{y}u$, and $\deg_{x}t<\deg_{x}u$.
\end{itemize}
\end{itemize}
\end{itemize}

It is routine to verify that this is a module ordering; it is obviously
not compatible with $<$.

Let
\[
f_{1}=x^{2}+xy,\quad f_{2}=xy+z^{2}.
\]
We have
\begin{equation}
\LT\left(\syz\left(\left\{ f_{1},f_{2}\right\} \right)\right)=\left\{ t\cdot\LT\left(f_{2}\sigvar_{1}-f_{1}\sigvar_{2}\right):t\in\T^{n}\right\} =\left\{ t\cdot z^{2}\sigvar_{1}:t\in\T^{n}\right\} .\label{eq: syzygy}
\end{equation}

For $i>3$, let $f_{i}$ be the $S$-polynomial of $f_{1}$ and $f_{i-1}$.
We have
\begin{align*}
f_{3} & =xy^{2}-xz^{2}\\
f_{4} & =xy^{3}+x^{2}z^{2}\\
 & \vdots\\
f_{i} & =xy^{i-1}+\left(-1\right)^{i}x^{i-2}z^{2}.
\end{align*}
For $i\geq2$, $\S\left(f_{i}\right)=x^{i-2}\sigvar_{2}=\min\left\{ t\sigvar_{2}:t\in\T^{n},\deg_{t}=i-2\right\} $.
So $f_{i}$ is not $\S$-reducible by $f_{2}$, \ldots{}, $f_{i-1}$
even though it is top-reducible by them. Additionally, $\LT\left(f_{1}\right)\nmid\LT\left(f_{i}\right)$,
so $f_{1}$ is $\S$-irreducible. 

It remains to see if the $f_{i}$ are primitive. Let $g\in I\backslash\left\{ 0\right\} $
such that $\LT\left(tg\right)=\LT\left(f_{i}\right)$ for some $t\in\T^{n}\backslash\left\{ 1\right\} $.
The factors of $\LT\left(f_{i}\right)$ are $x$ and $y$ \emph{only},
so $t=x^{a}y^{b}$ for some $a,b\in\mathbb{N}$. We also want $\S\left(tg\right)=\S\left(f_{i}\right)$;
we claim that this is possible only if $t=x^{a}$ for some $a\in\mathbb{N}$.
To see why, assume that $\S\left(tg\right)=\S\left(f\right)$. If
$t\S\left(g\right)\neq\S\left(g\right)$, we must have $t\S\left(g\right)\in\LT\left(\syz\left(\left\{ f_{1},f_{2}\right\} \right)\right)$,
which would imply that $t\S\left(g\right)=uz^{2}\sigvar_{1}$ for
some $u\in\T^{n}$. Since $\S\left(g\right)\in\NS\left(\syz\left(\left\{ f_{1},f_{2}\right\} \right)\right)$,
we infer that $2>\deg_{z}\S\left(g\right)=\deg_{z}\left(t\S\left(g\right)\right)$,
a contradiction to $t\S\left(g\right)=uz^{2}\sigvar_{1}$. Thus, $t\S\left(g\right)\in\NS\left(\syz\left(\left\{ f_{1},f_{2}\right\} \right)\right)$,
and $t\S\left(g\right)=\S\left(tg\right)$. So $t\S\left(g\right)=\S\left(tg\right)=\S\left(f_{i}\right)=x^{i-2}\sigvar_{2}$.
As claimed, $t=x^{a}$ for some $a\in\mathbb{N}$. Since $t\neq1$,
$a>0$; since $t\mid\LT\left(f_{i}\right)$, $a<2$. Hence, $t=x$
and $\LT\left(g\right)=y^{i-1}$, but a computation of the Gr\"obner
basis shows that $y^{i-1}\not\in\LT\left(I\right)=\left\langle y^{2}z^{2},xz^{2},x^{2},xy\right\rangle $.
Hence, there do not exist $g\in I\backslash\left\{ 0\right\} $ and
$t\in\T^{n}\backslash\left\{ 1\right\} $ such that $\LT\left(tg\right)=\LT\left(f_{i}\right)$
and $\S\left(tg\right)=\S\left(f_{i}\right)$.

It follows that any $\S$-Gr\"obner of $I$ with respect to $<$
and $<'$ must have all the polynomials $f_{2},f_{3},\ldots$, and
is thus infinite.
\end{example}
\bibliographystyle{plain}
\bibliography{The_F5_criterion_revised_v6}

\begin{thebibliography}{10}

\bibitem{SAGEImplementationOfF5}
Martin Albrecht and John Perry.
\newblock Implementation of {F}aug{\`e}re's {F}5 algorithm.
\newblock Sage library, 2008.

\bibitem{mb}
M.~Bardet.
\newblock {\em {\'E}tude des syst{\`e}mes alg{\'e}briques
  surd{\'e}termin{\'e}s. {A}pplications aux codes correcteurs et {\`a} la
  cryptographie}.
\newblock PhD thesis, LIP6, 2006.

\bibitem{GBBIB706}
B.~Buchberger.
\newblock {\em Ein Algorithmus zum Auffinden der Basiselemente des
  Restklassenringes nach einem nulldimensionalen Polynomideal}.
\newblock PhD thesis, University of Innsbruck, 1965.

\bibitem{Buchberger79}
Bruno Buchberger.
\newblock A criterion for detecting unnecessary reductions in the construction
  of {G}r{\"o}bner bases.
\newblock In E.~W. Ng, editor, {\em Proceedings of the EUROSAM 79 Symposium on
  Symbolic and Algebraic Manipulation, Marseille, June 26-28, 1979}, volume~72
  of {\em Lecture Notes in Computer Science}, pages 3--21, Berlin - Heidelberg
  - New York, 1979. Springer.

\bibitem{f5}
J.~C. Faug{\`e}re.
\newblock A new efficient algorithm for computing {G}r{\"o}bner bases without
  reduction to zero (${F}_5$).
\newblock In {\em ISSAC '02: Proceedings of the 2002 International Symposium on
  Symbolic and Algebraic Computation}, pages 75--83, New York, NY, USA, 2002.
  ACM Press.

\bibitem{Fau99}
Jean-Charles Faug{\`e}re.
\newblock A new efficient algorithm for computing {G}r{\"o}bner bases
  (${F}_4$).
\newblock {\em Journal of Pure and Applied Algebra}, 139(1--3):61--88, June
  1999.

\bibitem{Gash2008}
Justin Gash.
\newblock {\em On Efficient Computation of {G}r{\"o}bner Bases}.
\newblock Ph.{D}.~dissertation, Indiana University, Bloomington, IN, 2008.

\bibitem{StaggeredLinearBases}
Rudiger Gebauer and Hans M{\"o}ller.
\newblock Buchberger's algorithm and staggered linear bases.
\newblock In {\em Proceedings of SYMSAC 1986 (Waterloo/Ontario)}, pages
  218--221. ACM Press, 1986.

\bibitem{GBBIB1064}
Rudiger Gebauer and Hans M{\"o}ller.
\newblock On an installation of {B}uchberger's algorithm.
\newblock {\em Journal of Symbolic Computation}, 6:275--286, 1988.

\bibitem{CompCommAlgI}
Martin Kreuzer and Lorenzo Robbiano.
\newblock {\em Computational Commutative Algebra}, volume~1.
\newblock Springer, 2000.

\bibitem{Lazard83}
Daniel Lazard.
\newblock Gr{\"o}bner bases, {G}aussian elimination, and resolution of systems
  of algebraic equations.
\newblock In J.~A. {van Hulzen}, editor, {\em EUROCAL '83, European Computer
  Algebra Conference}, volume 162, pages 146--156. Springer LNCS, 1983.

\bibitem{143343}
Hans M{\"o}ller, Teo Mora, and Carlo Traverso.
\newblock Gr{\"o}bner bases computation using syzygies.
\newblock In {\em ISSAC '92: Proceedings of the International Symposium on
  Symbolic and Algebraic Computation}, pages 320--328, New York, NY, USA, 1992.
  ACM.

\bibitem{MoraPosso}
Teo Mora.
\newblock {\em Solving Polynomials Systems II: Macaulay's Paradigm and
  {G}r{\"o}bner Technology}.
\newblock Cambridge University Press, 2005.

\bibitem{SAGEImplementationOfSLB}
John Perry.
\newblock Implementation of {S}taggered {L}inear {B}asis algorithm.
\newblock Sage library, 2008.

\bibitem{SAGEImplementationOfF5Arri}
John Perry.
\newblock Implementation of {A}rri's {F}5 variant.
\newblock Sage library, 2010.

\bibitem{cryptoeprint:2006:404}
Till Stegers.
\newblock Faugere's {F}5 algorithm revisited.
\newblock Cryptology ePrint Archive, Report 2006/404, 2006.
\newblock \url{http://eprint.iacr.org/}.

\bibitem{sage}
William Stein.
\newblock {\em {Sage}: {O}pen {S}ource {M}athematical {S}oftware ({V}ersion
  4.1.1)}.
\newblock The Sage~Group, 2010.
\newblock {\tt www.sagemath.org}.

\end{thebibliography}

\end{document}